\documentclass[12pt]{article}
\usepackage{amssymb,amsthm,amsmath,latexsym}
\usepackage{url}
\usepackage{color}
\usepackage{comment}
\usepackage{lscape}
\newtheorem{thm}{Theorem}
\newtheorem{prop}[thm]{Proposition}
\newtheorem{lem}[thm]{Lemma}

\theoremstyle{remark}

\newcommand{\FF}{\mathbb{F}}
\newcommand{\ZZ}{\mathbb{Z}}

\DeclareMathOperator{\wt}{wt}

\begin{document}
\title{Ternary extremal four-negacirculant self-dual codes}

\author{
Masaaki Harada\thanks{
Research Center for Pure and Applied Mathematics,
Graduate School of Information Sciences,
Tohoku University, Sendai 980--8579, Japan.
email: \texttt{mharada@tohoku.ac.jp}},
Keita Ishizuka\thanks{
Research Center for Pure and Applied Mathematics,
Graduate School of Information Sciences,
Tohoku University, Sendai 980--8579, Japan.
email: \texttt{keita.ishizuka.p5@dc.tohoku.ac.jp}}
and
Hadi Kharaghani\thanks{
Department of Mathematics and Computer Science,
University of Lethbridge,
Lethbridge, AB, Canada T1K 3M4.
email: \texttt{kharaghani@uleth.ca}}
}

\maketitle

\begin{abstract}
In this note, we give basic properties of ternary four-negacirculant self-dual codes.
By exhaustive computer search based on the properties,
we complete a classification of ternary extremal four-negacirculant self-dual codes 
of lengths $40,44,48,52$ and $60$.
\end{abstract}

\section{Introduction}\label{sec:Intro}

Self-dual codes are one of the most interesting classes of codes.
This interest is justified by many combinatorial objects
and algebraic objects related to self-dual codes
(see e.g., \cite{RS-Handbook}).

A ternary self-dual code $C$ of length $n$
is an $[n,n/2]$ code over the finite field of
order $3$ satisfying
$C=C^\perp$, where $C^\perp$ is the dual code  of $C$.
A ternary self-dual code  of length $n$ exists if and only if $n$ is divisible by four (see~\cite{MPS}).
It was shown in~\cite{MS-bound} that
the minimum weight $d$ of a ternary self-dual code of length $n$
is bounded by $d\leq 3 \lfloor n/12 \rfloor+3$.
If $d=3\lfloor n/12 \rfloor+3$, then the ternary self-dual code is called {extremal}.
For $n \in \{4,8,12,\ldots,64\}$,
it is known that there is a ternary extremal self-dual code of length $n$ (see~\cite[Table~6]{Huffman}).
All ternary (extremal) self-dual codes of lengths up to $24$ and all ternary extremal self-dual codes of length $28$ 
were classified in~\cite{CPS}, \cite{HM}, \cite{MPS} and~\cite{PSW}.
Many four-negacirculant self-dual codes having large minimum weights are known.
In~\cite{HHKK}, all ternary extremal four-negacirculant self-dual codes of lengths
$32$ and $36$ were classified.
The aim of this note is to extend the classification to lengths $40,44,48,52$ and $60$.

This note is organized as follows.
In Section~\ref{Sec:2},
we give definitions and some known results
of ternary self-dual codes used in this note.
In Section~\ref{Sec:B}, basic properties of ternary four-negacirculant self-dual codes are given.
In Section~\ref{Sec:C},
by exhaustive computer search based on the properties in Section~\ref{Sec:B}, we give a complete classifications of ternary extremal four-negacirculant self-dual codes 
of lengths $40,44,48,52$ and $60$.

\section{Preliminaries}\label{Sec:2}

In this section, we give definitions and some known results
of ternary self-dual codes used in this note.

\subsection{Ternary self-dual codes}
Let $\FF_3=\{0,1,2\}$ denote the finite field of order $3$.
A \emph{ternary} $[n,k]$ \emph{code} $C$ is a $k$-dimensional vector subspace
of $\FF_3^n$.
All codes in this note are ternary, and the expression ``codes'' means ``ternary codes''.
The parameter $n$ is called the \emph{length} of $C$.
A \emph{generator matrix} of $C$
is a $k \times n$ matrix whose rows are a basis of $C$.
The \emph{weight} $\wt(x)$ of a vector $x$ of $\FF_3^n$ is 
the number of non-zero components of $x$.
A vector of $C$ is called a \emph{codeword}.
The minimum non-zero weight of all codewords in $C$ is called
the \emph{minimum weight} of $C$.

The {\em dual} code $C^{\perp}$ of a code
$C$ of length $n$ is defined as
$
C^{\perp}=
\{x \in \FF_3^n \mid x \cdot y = 0 \text{ for all } y \in C\},
$
where $x \cdot y$ is the standard inner product.
A code $C$ is \emph{self-dual} if $C=C^\perp$.
A self-dual code  of length $n$ exists if and only if $n$ is divisible by four.
Two codes $C$ and $C'$ are \emph{equivalent} if there is a
monomial matrix $P$ over $\mathbb{F}_3$ with $C' = C \cdot P$,
where $C \cdot P = \{ x P\mid  x \in C\}$.
We denote two equivalent codes $C$ and $D$ by $C \cong D$.
All self-dual codes were classified in~\cite{CPS}, \cite{HM}, \cite{MPS} 
and~\cite{PSW} for lengths up to $24$.

It was shown in~\cite{MS-bound} that
the minimum weight $d$ of a self-dual code of length $n$
is bounded by $d\leq 3 \lfloor n/12 \rfloor+3$.
If $d=3\lfloor n/12 \rfloor+3$, then the code is called \emph{extremal}.
For $n \in \{4,8,12,\ldots,64\}$,
it is known that there is an extremal self-dual code of length $n$ (see~\cite[Table~6]{Huffman}).
All extremal self-dual codes of length $28$ were classified in~\cite{HM}.

\subsection{Ternary four-negacirculant codes}

An $n \times n$  matrix of the following form
\[
\left( \begin{array}{cccccc}
r_0&r_1&r_2& \cdots &r_{n-2} &r_{n-1}\\
2 r_{n-1}&r_0&r_1& \cdots &r_{n-3}&r_{n-2} \\
2 r_{n-2}&2 r_{n-1}&r_0& \cdots &r_{n-4}&r_{n-3} \\
\vdots &\vdots & \vdots &&\vdots& \vdots\\
\vdots &\vdots & \vdots &&\vdots& \vdots\\
2 r_1&2 r_2&2 r_3& \cdots&2r_{n-1}&r_0
\end{array}
\right)
\]
is called  \emph{negacirculant}.
If  $A$ and $B$ are $n \times n$ negacirculant matrices, then
a $[4n,2n]$ code with generator matrix of the following form
\begin{equation}\label{eq:4nega}
\left(
\begin{array}{ccc@{}c}
\quad & {I_{2n}} & \quad &
\begin{array}{cc}
A & B \\
2B^T & A^T
\end{array}
\end{array}
\right)
\end{equation}
is called a \emph{four-negacirculant} code.
Throughout this note, we denote the four-negacirculant code with generator matrix of the
form~\eqref{eq:4nega} by $C(A,B)$ or $C(r_A,r_B)$, where
$r_A$ and $r_B$ are the first rows of $A$ and $B$, respectively.

Note that $2A$, $A^T$ and $2A^T$ are negacirculant matrices for a negacirculant matrix $A$.
Let $A$ and $B$ be  negacirculant matrices.
Then four-negacirculant  codes 
$C(\alpha A,\beta B)$, 
$C(\alpha A,\beta B^T)$, 
$C(\alpha A^T,\beta B)$ and  
$C(\alpha A^T,\beta B^T)$ 
are defined for $\alpha, \beta \in \{1,2\}$.
It is trivial that $C(A,B)$ is self-dual if $AA^T+BB^T=2I_n$~\cite{HHKK}.
%

Many four-negacirculant self-dual codes having large minimum weights are known 
(see e.g., \cite{HHKK}). 
In~\cite{HHKK}, all extremal four-negacirculant self-dual codes of lengths
$32$ and $36$ were classified.
In this note, we extend the classification to lengths $40,44,48,52$ and $60$.
In order to do this, 
we give basic properties of four-negacirculant self-dual codes in the next section.

\section{Basic properties of ternary four-negacirculant self-dual codes}\label{Sec:B}

In this section,
we give basic properties of ternary four-negacirculant self-dual codes 
for a classification of ternary four-negacirculant self-dual codes 
of a given length.

\begin{lem}\label{lem:triv_equiv}
  Let $C(A,B)$ be a ternary four-negacirculant code.
  Then the following statements hold:
  \begin{enumerate}
    \item $C(A,B) \cong C(B,A)$,
    \item $C(A,B) \cong C(A,2B)$,
    \item If $C(A,B)$ is self-dual, then $C(A,B) \cong C(A,B^T)$.

  
      \end{enumerate}
\end{lem}
\begin{proof}
We only give a proof for (iii) as the other two cases are trivial.
If $C(A,B)$ is self-dual, then the following parity check matrix of $C(A,B)$
\begin{equation*}
\left(
\begin{array}{ccc@{}c}
\begin{array}{cc}
2A^T & B \\
2B^T & 2A
\end{array}
\quad & {I_{2n}} & \quad
\end{array}
\right)
\end{equation*}
is also a generator matrix of $C(A,B)$.
This yields that $C(A,B) \cong C(2A,2B^T) \cong C(A,B^T)$.
This completes the proof.
\end{proof}

For $a=(a_1,a_2,\dots,a_n) \in \FF_3^n$, define a map $f$ from $\FF_3^n$ to $\ZZ$ as
$f(a)=\sum_{i=1}^n 3^{i-1} a_i$, regarding $a_i=0,1$ or $2$ as an element of $\ZZ$.
Then, for $a,b \in \FF_3^n$,
we define an ordering $a \ge b$ based on the natural ordering of $\ZZ$
as $f(a) \ge f(b)$.

\begin{lem} \label{lem:reduction}
  Let $C(r_A,r_B)$ be a ternary four-negacirculant code.
   Then there is a ternary four-negacirculant code $C(r_{A'},r_{B'})$
satisfying the following conditions:
\begin{itemize}
  \item[\rm (C1)] $C(r_A,r_B) \cong C(r_{A'},r_{B'})$,
  \item[\rm (C2)] the first nonzero components of $r_{A'}$ and $r_{B'}$ are $1$,
  \item[\rm (C3)] $r_{A'} \ge r_{B'}$.
\end{itemize}
\end{lem}
\begin{proof}
By Lemma~\ref{lem:triv_equiv} (i) and (ii), it follows that
\begin{equation} \label{eq:foureq}
     C(r_A, r_B) \cong C(r_A, 2r_B)\cong C(2r_A, r_B) \cong C(2r_A, 2r_B).
\end{equation}
Let $a$ and $b$ denote the first nonzero components of $r_A$ and $r_B$, respectively.
It follows from~\eqref{eq:foureq} that $C(r_A, r_B) \cong C(a\cdot r_A, b\cdot r_B)$.
Note that the first nonzero components of $a\cdot r_A$ and $b\cdot r_B$ are $1$.
If $a\cdot r_A < b\cdot r_B$, then we can take $C(b\cdot r_B, a\cdot r_A)$ instead of 
$C(a\cdot r_A, b\cdot r_B)$ since $C(r_A, r_B) \cong C(r_B, r_A)$.
This completes the proof.
\end{proof}

By the above lemma,
it is sufficient to consider only ternary four-negacirculant self-dual codes $C(r_A,r_B)$
satisfying the conditions (C1)--(C3), in order to complete 
a classification of  ternary four-negacirculant self-dual codes of a fixed length.
This substantially reduces the number of such codes which need be checked for equivalences.

\section{Ternary extremal four-negacirculant self-dual codes}\label{Sec:C}

In this section, we give a classification of ternary extremal four-negacirculant self-dual codes 
of lengths $40,44,48,52$ and $60$ by finding all distinct ternary extremal 
four-negacirculant self-dual codes 
satisfying the conditions (C1)--(C3) in Lemma~\ref{lem:reduction}.
This classification was done by exhaustive computer search and 
all computer calculations in this section
were done with the help of \textsc{Magma}~\cite{Magma}.

\subsection{Length 40}

Our exhaustive computer search found all distinct ternary extremal four-negacirculant self-dual
codes of length $40$
satisfying the conditions (C1)--(C3) in Lemma~\ref{lem:reduction}.
By equivalence testing, 
we have the following:

\begin{prop}\label{prop:40}
There are $116$ inequivalent ternary extremal four-negacirculant self-dual codes of length $40$.
\end{prop}

In~\cite[Table~3]{HHKK},
$100$ inequivalent ternary extremal four-negacirculant self-dual codes of this length
are found.
For the remaining $16$ codes $C_{40,i}=C(r_A,r_B)$ $(i=1,2,\ldots,16)$, 
the first rows $r_A$ and $r_B$ are listed in Table~\ref{Tab:40}.

\begin{table}[thb]
\caption{Extremal four-negacirculant self-dual codes of length $40$}
\label{Tab:40}
\centering
\medskip
{\footnotesize
\begin{tabular}{c|c|c}
\noalign{\hrule height1pt}
$i$ & $r_A$ & $r_B$ \\
  \hline
 $1$ & $(0,0,0,0,1,1,1,1,1,0)$ & $(1,1,2,1,0,0,1,0,1,0)$ \\
 $2$ & $(0,0,0,0,1,1,2,0,0,1)$ & $(1,1,0,2,1,2,0,1,1,0)$ \\
 $3$ & $(0,0,0,0,1,2,2,0,1,1)$ & $(1,1,1,0,0,2,0,0,1,1)$ \\
 $4$ & $(0,0,0,0,0,1,1,0,2,1)$ & $(1,1,1,2,1,1,2,1,1,1)$ \\
 $5$ & $(0,0,0,1,0,1,1,1,0,1)$ & $(1,2,0,0,1,1,2,0,0,1)$ \\
 $6$ & $(0,0,0,1,2,2,1,1,0,1)$ & $(1,1,2,1,2,2,0,2,2,0)$ \\
 $7$ & $(0,0,0,1,1,1,1,2,0,1)$ & $(1,2,1,2,2,0,1,1,2,0)$ \\
 $8$ & $(0,0,0,1,2,0,0,0,1,1)$ & $(1,2,2,1,1,0,2,0,1,0)$ \\
 $9$ & $(0,0,0,1,1,2,1,0,1,1)$ & $(1,2,1,1,0,1,1,0,1,1)$ \\
$10$ & $(0,0,0,1,1,1,2,0,1,1)$ & $(1,1,1,1,2,2,0,1,1,0)$ \\
$11$ & $(0,0,0,1,1,2,0,0,1,2)$ & $(1,0,2,1,1,1,0,0,1,0)$ \\
$12$ & $(0,0,0,1,1,1,1,2,1,2)$ & $(1,0,1,2,0,1,1,0,1,2)$ \\
$13$ & $(0,0,1,1,1,1,1,1,2,1)$ & $(1,1,1,2,2,1,1,0,2,1)$ \\
$14$ & $(0,1,1,1,1,0,1,2,2,1)$ & $(1,1,1,2,0,1,1,1,1,1)$ \\
$15$ & $(1,1,1,1,2,1,2,1,1,1)$ & $(1,2,2,1,1,2,1,1,1,1)$ \\
$16$ & $(1,1,1,2,1,1,2,1,1,1)$ & $(1,1,2,1,2,1,1,1,1,1)$ \\
\noalign{\hrule height1pt}
\end{tabular}
}
\end{table}

By the Assmus and Mattson theorem~\cite{AM}, the supports of $19760$ codewords of minimum
weight in a ternary extremal self-dual code of length $40$ form a $3$-$(40,12,220)$ design.
We verified that the  $116$ $3$-$(40,12,220)$ designs obtained from the above codes
are non-isomorphic.
This also confirms the inequivalences of the above $116$ codes.

\subsection{Length 44}

Our exhaustive computer search found all distinct ternary extremal four-negacirculant self-dual
codes of length $44$
satisfying the conditions (C1)--(C3) in Lemma~\ref{lem:reduction}.
By equivalence testing, 
we have the following:

\begin{prop}\label{prop:44}
There are $1518$ inequivalent ternary extremal four-negacirculant self-dual codes of length $44$.
\end{prop}

The first rows $r_A$ and $r_B$ of the $1518$ codes $C_{44,i}=C(r_A,r_B)$ $(i=1,2,\ldots,1518)$
can be obtained electronically from
\url{https://www.math.is.tohoku.ac.jp/~mharada/F3-44.txt}.

\subsection{Length 48}

Two inequivalent ternary extremal four-negacirculant self-dual codes of length $48$
are found in~\cite{HHKK}.  
Our exhaustive computer search verified that every ternary extremal four-negacirculant self-dual
code of length $48$ satisfying the conditions (C1)--(C3) in Lemma~\ref{lem:reduction}
is equivalent to one of the two codes in~\cite{HHKK}.
Thus, we have the following:

\begin{prop}\label{prop:48}
There are two inequivalent ternary extremal four-negacirculant self-dual codes of length $48$.
\end{prop}

\subsection{Length 52}
Two inequivalent  ternary extremal four-negacirculant self-dual codes of length $52$
are found in~\cite{HHKK}.  
Our exhaustive computer search verified that every ternary extremal four-negacirculant self-dual
code of length $52$ satisfying the conditions (C1)--(C3) in Lemma~\ref{lem:reduction}
is equivalent to one of the two codes in~\cite{HHKK}.
Thus, we have the following:

\begin{prop}\label{prop:52}
There are two inequivalent ternary extremal four-negacirculant self-dual codes of length $52$.
\end{prop}

\subsection{Length 60}

Three inequivalent  ternary extremal four-negacirculant self-dual codes of length $60$
are found in~\cite{AHM}.  
Our exhaustive computer search verified that every ternary extremal four-negacirculant self-dual
code of length $60$ satisfying the conditions (C1)--(C3) in Lemma~\ref{lem:reduction}
is equivalent to one of the three codes in~\cite{AHM}.
Thus, we have the following:

\begin{prop}\label{prop:60}
There are three inequivalent ternary extremal four-negacirculant self-dual codes of length $60$.
\end{prop}

\subsection{Other lengths}
\label{Sec:Other}

Our exhaustive computer search found all distinct ternary extremal four-negacirculant self-dual
codes of length $56$
satisfying the conditions (C1)--(C3) in Lemma~\ref{lem:reduction}.
Then $14331912$ codes were obtained.
It seems that testing equivalences of a large number of  ternary extremal four-negacirculant self-dual
codes of this length is  beyond our current computer resources.
We completed  testing equivalences of ternary extremal four-negacirculant self-dual
codes $C(A,B)$ satisfying the condition that
the first five coordinates of the first row $r_A$ of $A$
are
\[
(0,0,0,0,0),
(0,0,0,0,1),
(0,0,0,1,0),
(0,0,0,1,1),
(0,0,0,1,2).
\]
Then $3001$ inequivalent  ternary extremal four-negacirculant self-dual
codes are obtained.
The first rows $r_A$ and $r_B$ of the $3001$ codes $C_{56,i}=C(r_A,r_B)$ $(i=1,2,\ldots,3001)$
can be obtained electronically from
\url{https://www.math.is.tohoku.ac.jp/~mharada/F3-56.txt}.

\begin{prop}\label{prop:56}
There are at least $3001$ inequivalent ternary extremal four-negacirculant self-dual codes of length $56$.
\end{prop}

A ternary extremal four-negacirculant self-dual code of length $64$
is found in~\cite{HHKK}.  
For length $n \ge 64$,
it seems that an exhaustive computer search is beyond our current computer resources.
Our non-exhaustive computer search
failed to discover any of a new 
ternary extremal four-negacirculant self-dual code of length $64$ and
a ternary extremal four-negacirculant self-dual code of length $68$.
Here we give a ternary four-negacirculant self-dual code $C(r_A,r_B)$ of length $68$ and minimum
weight $15$, where
\begin{align*}
r_A&=( 0, 2, 2, 0, 2, 0, 1, 1, 0, 1, 0, 0, 0, 1, 2, 1, 0), \\
r_B&=( 1, 0, 0, 2, 2, 1, 0, 2, 2, 1, 1, 0, 1, 0, 2, 1, 0).
\end{align*}
Note that a ternary extremal self-dual code of length $68$ is currently not known.

\subsection{Summary}

We end this note by listing a summary of a classification of 
ternary extremal four-negacirculant self-dual codes.
For $n\in \{32,36,\ldots,68\}$,
the number $N(n)$ of inequivalent ternary extremal four-negacirculant self-dual codes
of length $n$ is listed in Table~\ref{Tab:Summary}, along with the references.

\begin{table}[thb]
\caption{Summary}
\label{Tab:Summary}
\centering
\medskip
{\small
\begin{tabular}{c|c|c}
\noalign{\hrule height1pt}
$n$ & $N(n)$  & References\\
    \hline
$32$ & $ 53$ & \cite{HHKK}\\
$36$ & $ 1$ & \cite{HHKK}\\
$40$ & $116 $ & Proposition~\ref{prop:40}\\
$44$ & $1518 $ & Proposition~\ref{prop:44}\\
$48$ & $2$ & Proposition~\ref{prop:48}\\
$52$ & $2 $ & Proposition~\ref{prop:52}\\
$56$ & $\ge 3001$ & Proposition~\ref{prop:56}\\
$60$ & $3$ & Proposition~\ref{prop:60}\\
$64$ & $\ge 1$ & \cite{HHKK} \\
$68$ & ? & \\
\noalign{\hrule height1pt}
\end{tabular}
}
\end{table}

\bigskip
\noindent
\textbf{Acknowledgments.}
Masaaki Harada is supported by JSPS KAKENHI Grant Numbers 19H01802 and 23H01087.
Hadi Kharaghani is supported by the Natural Sciences and Engineering Research
Council of Canada (NSERC).



\end{document}